\documentclass[11pt,leqno]
{amsart}
\usepackage{amsmath,epsfig,graphicx,color}

\usepackage{dsfont}
\usepackage{ifthen}
\usepackage{amssymb,latexsym}
\usepackage{subfigure}
\usepackage{hyperref}
\hypersetup{
urlcolor=black, 
  menucolor=black, 
  citecolor=black, 
  anchorcolor=black, 
  filecolor=black, 
  linkcolor=black, 
  colorlinks=true,
}
\usepackage{comment}
\usepackage[numbers,sort&compress]{natbib} 
\newcommand{\cip}{\stackrel{\P}{\rightarrow}}
\newcommand{\cas}{\stackrel{\rm a.s.}{\rightarrow}}

\newcommand{\dint}{\,\mathrm{d}}
\newcommand{\x}{\mathbf{x}}
\newcommand{\X}{\mathbf{X}}
\newcommand{\limn}{\lim_{n \to \infty}}

\definecolor{darkblue}{rgb}{.2, 0.2,.8}
\definecolor{darkgreen}{rgb}{0,0.5,0.3}
\definecolor{darkred}{rgb}{.8, .1,.1}

\newcommand{\blue}{\color{darkblue}}

\newcommand{\ex}{{\rm e}\,}

\newcommand{\asy}{asymptotic}

\newtheorem{lemma}{Lemma}[section]

\newtheorem{theorem}[lemma]{Theorem}
\newcommand{\Leb}{{\mathbb L}{\mathbb E}{\mathbb B}}

\newtheorem{proposition}[lemma]{Proposition}
\newtheorem{definition}[lemma]{Definition}
\newtheorem{corollary}[lemma]{Corollary}
\newtheorem{example}[lemma]{Example}
\newtheorem{exercise}[lemma]{Exercise}
\newtheorem{remark}[lemma]{Remark}
\newtheorem{fig}[lemma]{Figure}
\newtheorem{tab}[lemma]{Table}

\newcommand{\cid}{\stackrel{d}{\rightarrow}}

\newcommand{\bfR}{{\bf R}}

\newcommand{\bth}{\begin{theorem}}
\newcommand{\ethe}{\end{theorem}}

\newcommand{\bre}{\begin{remark}\em }
\newcommand{\ere}{\end{remark}}

\newcommand{\ble}{\begin{lemma}}
\newcommand{\ele}{\end{lemma}}

\newcommand{\pp}{point process}
\newcommand{\bde}{\begin{definition}}
\newcommand{\ede}{\end{definition}}
\newcommand{\bco}{\begin{corollary}}
\newcommand{\eco}{\end{corollary}}

\newcommand{\bpr}{\begin{proposition}}
\newcommand{\epr}{\end{proposition}}

\newcommand{\bexer}{\begin{exercise}}
\newcommand{\eexer}{\end{exercise}}

\newcommand{\bexam}{\begin{example}}
\newcommand{\eexam}{\end{example}}

\newcommand{\bfi}{\begin{fig}}
\newcommand{\efi}{\end{fig}}

\newcommand{\btab}{\begin{tab}}
\newcommand{\etab}{\end{tab}}

\newcommand{\rv}{random variable}

\newcommand{\var}{{\rm var}}

\newcommand{\as}{{\rm a.s.}}

\newcommand{\rhs}{right-hand side}
\newcommand{\df}{distribution function}

\newcommand{\beao}{\begin{eqnarray*}}
\newcommand{\eeao}{\end{eqnarray*}\noindent}

\newcommand{\beam}{\begin{eqnarray}}
\newcommand{\eeam}{\end{eqnarray}\noindent}

\newcommand{\beqq}{\begin{equation}}
\newcommand{\eeqq}{\end{equation}\noindent}

\newcommand{\bce}{\begin{center}}
\newcommand{\ece}{\end{center}}

\newcommand{\diag}{{\rm diag}}
\newcommand{\barr}{\begin{array}}
\newcommand{\earr}{\end{array}}

\newcommand{\stp}{\stackrel{\P}{\rightarrow}}
\newcommand{\std}{\stackrel{d}{\rightarrow}}

\newcommand{\eqd}{\stackrel{d}{=}}

\newcommand{\vague}{\stackrel{\lower0.2ex\hbox{$\scriptscriptstyle
                    \it{v} $}}{\rightarrow}}
\newcommand{\weak}{\stackrel{\lower0.2ex\hbox{$\scriptscriptstyle
                    \it{w} $}}{\rightarrow}}
\newcommand{\what}{\stackrel{\lower0.2ex\hbox{$\scriptscriptstyle
                    \it{\hat{w}} $}}{\rightarrow}}

\newcommand{\bdis}{\begin{displaymath}}
\newcommand{\edis}{\end{displaymath}\noindent}

\newcommand{\N}{\mathbb{N}}
\newcommand{\R}{\mathbb{R}}

\newcommand{\nto}{n\to\infty}
\newcommand{\kto}{k\to\infty}
\newcommand{\xto}{x\to\infty}

\newcommand{\ov}{\overline}
\newcommand{\wt}{\widetilde}
\newcommand{\wh}{\widehat}
\newcommand{\vep}{\varepsilon}

\newcommand{\regvary}{regularly varying}
\newcommand{\slvary}{slowly varying}
\newcommand{\regvar}{regular variation}

\newcommand{\bbr}{{\mathbb R}}

\newcommand{\con}{convergence}

\newcommand{\st}{such that}
\newcommand{\fif}{if and only if}

\newcommand{\fct}{function}

\newcommand{\ds}{distribution}

\newcommand{\cmt}{continuous mapping theorem}
\newcommand{\seq}{sequence}

\newcommand{\pro}{probabilit}

\newcommand{\ms}{measure}
\newcommand{\mgf}{moment generating function}
\newcommand{\ld}{large deviation}

\newcommand{\bfX}{{\bf X}}

\newcommand{\bfS}{{\bf S}}

\newcommand{\bfI}{{\bf I}}

\newcommand{\E }{{\mathbb E}}
\renewcommand{\P }{{\mathbb P}}

\newcommand{\1}{\mathds{1}}

\allowdisplaybreaks

\DeclareMathOperator{\e}{e}

\newcommand{\MP}{Mar\v cenko--Pastur }

\newcommand{\norm}[1]{\|#1\|}

\begin{document}
\today
\bibliographystyle{acm}
\title[The off-diagonal point process]{Point process convergence for the off-diagonal entries of  sample covariance matrices}
\thanks{Thomas Mikosch's research is partly support by an Alexander von 
Humboldt Research Award. He takes pleasure in thanking the Faculty of
Mathematics of Ruhruniversit\"at Bochum for hosting him in the period December 2018--May 2019. The research of TM and Jorge Yslas is also supported
by Danmarks Frie Forskningsfond Grant No 9040-00086B.\\
Johannes Heiny was supported by the Deutsche Forschungsgemeinschaft (DFG) via RTG 2131 High-dimensional Phenomena in Probability - Fluctuations and Discontinuity.}

\author[J. Heiny]{Johannes Heiny}
\address{Fakult\"at f\"ur Mathematik,
Ruhruniversit\"at Bochum,
Universit\"atsstrasse 150,
D-44801 Bochum,
Germany}
\email{johannes.heiny@rub.de}
\author[T. Mikosch]{Thomas Mikosch}
\address{Department  of Mathematics,
University of Copenhagen,
Universitetsparken 5,
DK-2100 Copenhagen,
Denmark}
\email{mikosch@math.ku.dk}
\author[J. Yslas]{Jorge Yslas}
\address{Department  of Mathematics,
University of Copenhagen,
Universitetsparken 5,
DK-2100 Copenhagen,
Denmark}
\email{jorge.yslas1@gmail.com}

\begin{abstract}
We study  \pp\ \con\ for \seq s of iid random walks. The objective
is to derive \asy\ theory for the extremes of these random walks.
We show \con\ of the maximum random walk to the Gumbel \ds\ under the existence 
of a $(2+\delta)$th moment. We make heavily use of precise \ld\
results for sums of iid \rv s. As a con\seq , we derive the joint 
convergence of the off-diagonal entries in sample covariance and 
correlation matrices of a high-dimensional sample 
whose dimension increases with the sample size.
This generalizes known results on the \asy\ Gumbel property of 
the largest entry.
\end{abstract}
\maketitle

\section{Introduction}\label{sec:intro}

\subsection{Motivation}
An accurate probabilistic understanding of covariances and correlations is often the backbone of a thorough statistical data analysis. In many contemporary applications, one is 
faced with large data sets where both the dimension of the observations and the sample size are large. 
A major reason lies in the rapid improvement of computing power and data collection devices which has triggered the necessity to study and interpret the sometimes overwhelming amounts of data in an efficient and tractable way. Huge data sets arise naturally in genome sequence data in biology, online networks, wireless communication, large financial portfolios, and natural sciences. More applications where the dimension $p$ might be of the same or even higher magnitude than the sample size $n$ are discussed in \cite{donoho:2000,johnstone:2001}. In such a high-dimensional setting, one faces new probabilistic and statistical challenges; see \cite{johnstone:titterington:michael:2009} for a review. 
The sample (auto)covariance matrices will typically be misleading \cite{bai:silverstein:2010,elkaroui:2009}. Even in the null case, i.e.,~when the components of the time series are iid, it is well-known that the sample covariance matrix poorly estimates the population covariance matrix. The fluctuations of the off-diagonal entries of the sample covariance matrix aggregate, creating an estimation bias which is quantified by the famous \MP~theorem  \cite{marchenko:pastur:1967}. This paper provides insight into the joint behavior of the off-diagonal entries with a particular focus on their extremes.

Aside from the high dimension, the marginal distributions of the components present another major challenge for an accurate assessment of the dependence. In the literature, one typically assumes a finite fourth moment since otherwise the largest eigenvalue of the sample covariance matrix would tend to infinity when $n$ and $p$ increase. 
This moment assumption, however, excludes heavy-tailed time series from the analysis. The theory for the eigenvalues and eigenvectors of the sample autocovariance matrices stemming from such time series is quite different from the classical \MP~theory which applies in the light-tailed case.
For detailed discussions about classical random matrix theory, we refer to the monographs \cite{bai:silverstein:2010,yao:zheng:bai:2015}, 
while the developments in the heavy-tailed case can be found in \cite{auffinger:arous:peche:2009,davis:mikosch:pfaffel:2016,davis:heiny:mikosch:xie:2016,heiny:mikosch:2017:iid,soshnikov:2004,soshnikov:2006} and the references therein. 
 For applications of extreme value statistics in finance and physics we refer to \cite{bun2017cleaning,fyodorov2008freezing}.

In this paper, we study  point process convergence for sequences of iid random walks.
We then apply our results to derive the joint asymptotic behavior of the off-diagonal entries of sample covariance and correlation matrices. 
Based on this joint convergence we propose new independence tests in high dimensions.

\subsection{The model}\label{sec:model}

We are given $p$-dimensional random vectors  $\x_t=(X_{1t}, \ldots, X_{pt})^\top$, $ t=1,\ldots,n$, whose components $(X_{it})_{i,t\ge 1}$ satisfy the following {\em standard conditions}:
\begin{itemize}
\item $(X_{it})$ are independent and identically distributed (iid) random variables with generic element $X$. 
\item $\E[X]=0$ and $\E[X^2]=1$. 
\end{itemize}
The dimension $p=p_n$ is some integer sequence tending to infinity as $\nto$.

We are interested in the (non-normalized) $p\times p$ sample covariance matrix 
$\bfS$ and the sample correlation matrix $\bfR$,
\begin{equation}\label{eq:dcov}
	\bfS= \bfS_n= \sum_{t=1}^n \x_t \x_t^\top \quad  \text{ and } \quad \bfR = \bfR_n= (\diag(\bfS))^{-1/2} \bfS (\diag(\bfS))^{-1/2}
\end{equation}
with entries
\beam\label{eq:dcorr}
S_{ij}= \sum_{t=1}^n X_{it}X_{jt} \quad \text { and } \quad R_{ij} = \frac{S_{ij}}{\sqrt{S_{ii} S_{jj}}} \,, \quad i,j=1,\ldots,p\,,
\eeam
respectively.
The dependence on $n$ is often suppressed in our notation.
 
Our goal is to prove limit theory for the \pp es of scaled and centered
points $(S_{ij})$, $(R_{ij})$. Asymptotic theory for the extremes of these points can be deduced from the limit point process.

\subsection{State-of-the-art}
In the literature, the largest off-diagonal entry of a sample covariance or correlation matrix has been studied, but results concerning the joint behavior of the entries are currently lacking. The theoretical developments are mainly due to Jiang. In \cite{jiang:2004b}, he analyzed the asymptotic distributions of 
\begin{equation*}
W_n:= \frac{1}{n} \max_{1\le i<j \le p} |S_{ij}| \quad \text{ and } \quad L_n:= \max_{1\le i<j \le p} |R_{ij}|
\end{equation*}
under the assumption $p/n\to \gamma \in (0,\infty)$. If $\E[|X|^{30+\delta}]<\infty$ for some $\delta>0$, he proved that 
\begin{equation}\label{eq:jiang2}
\limn \P( n W_n^2-4 \log p + \log \log p \le x) = \exp\Big(-\tfrac{1}{\sqrt{8 \pi}} \e^{-x/2} \Big)\,, \quad x\in \R\,,
\end{equation}
\begin{equation}\label{eq:jiang1}
\limn \P( n L_n^2-4 \log p + \log \log p \le x) = \exp\Big(-\tfrac{1}{\sqrt{8 \pi}} \e^{-x/2} \Big)\,, \quad x\in \R\,.
\end{equation}
The limiting law is a non-standard Gumbel distribution. Under the same assumptions Jiang \cite{jiang:2004b} also derived the limits
\begin{equation}\label{eq:jiang3}
\limn \sqrt{\frac{n}{\log p}} L_n =2 = \limn \sqrt{\frac{n}{\log p}} W_n \qquad \as
\end{equation}
Several authors managed to relax Jiang's moment condition while keeping the proportionality of $p$ and $n$. Zhou \cite{zhou:2007} showed that \eqref{eq:jiang1} holds if $n^6 \P(|X_{11}X_{12}|>n)\to 0$ as $\nto$. A sufficient condition is $\E[X^6]<\infty$. The papers \cite{li:liu:rosalsky:2010,li:qi:rosalsky:2012,li:rosalsky:2006} provide refinements of Zhou's condition. We summarize the distributional assumptions on $X$ for the validity of  \eqref{eq:jiang1} and \eqref{eq:jiang2} under proportionality of dimension $p$ and sample size $n$ as follows: $\E[X^6]<\infty$ is sufficient, and $\E[|X|^{6-\delta}]<\infty$ for any $\delta>0$ is necessary. In that sense, finiteness of the sixth moment is the optimal moment assumption. 

Interestingly, the optimal moment requirement also depends on the growth of $p$ if $p$ increases at a different rate than $n$. 
For the largest off-diagonal entry of the sample correlation matrix, also known as coherence of the random matrix $\bfX=(\x_1, \ldots,\x_n)$, the interplay between dimension and moments was addressed in \cite{liu:lin:shao:2008}. If $\E[|X|^s]<\infty$ for $s>2$ and
\begin{equation*}
c_1 n^{(s-2)/4} \le p\le c_2 n^{(s-2)/4}
\end{equation*}
with positive constants $c_1,c_2$, Theorem 1.1 in \cite{liu:lin:shao:2008} shows that \eqref{eq:jiang1} still applies. Note that, for proportional $p$ and $n$, this result requires the finiteness of the sixth moment. The larger $p$ relative to $n$, the more moments of $X$ are needed. If the moment generating function of $|X|$ exists in some neighborhood of zero, \eqref{eq:jiang1} holds for $p=O\big(\exp(n^\beta)\big)$ for certain $\beta\in (0,1/3)$; see \cite{cai:jiang:2011}.
Finally, if $(\log p)/n \not\to 0$, various phase transitions appear in the limit distribution of $L_n$. These were explored in \cite{cai:jiang:2012} under convenient assumptions on $X$ which yield an explicit formula for the density of $R_{12}$. 

\subsection{Objective and structure of this paper} 
Our main objective is to prove limit theory for the \pp es of scaled and centered points $(S_{ij})$, $(R_{ij})$ in a more general framework than used for the results above. 
By a continuous mapping argument, the joint asymptotic distribution of functionals of a fixed number of points can easily be deduced from the limit process.  In particular, we obtain the asymptotic distribution of the largest and smallest entries. 

First, we establish our result for $\bfS$. Since each $S_{ij}$ is a sum of iid random variables, we prove a useful large deviation theorem which exploits 
the asymptotic normal distribution of $S_{ij}$. 
Aside from finding suitable assumptions on $X$, the main challenge is that the $S_{ij}$ are not independent.  It turns out that despite their non-trivial dependence, the maximum behaves like the maximum of iid copies.
Therefore we will first solve the problem for iid random walk points $(S_n^{(i)})$ instead of $(S_{ij})$. This is done in Section~\ref{sec:largedev}.
\par
We continue in Section~\ref{sec:main} with the main results of the paper. Here
we derive \asy\ theory for the \pp es
\beao
\wt N_n= \sum_{1\le i<j\le p} \vep_{\wt d_p (S_{ij}/\sqrt{n}-\wt d_p)}\std N\,,
\eeao
for suitable constants $(\wt d_p)$ and limit Poisson random \ms\ $N$ 
with mean \ms\ $\mu$ on $\bbr$ \st\ $\mu(x,\infty)=\ex^{-x}$, $x\in\bbr$. 
Throughout this paper $\vep_x$ denotes the Dirac measure at $x$.
A continuous mapping theorem implies \ds al \con\ of finitely many
$S_{ij}$. In particular, the maximum entry $S_{ij}$ converges to the
Gumbel \ds\ provided $X$ has suitably many finite moments.
A related result holds
with $S_{ij}$ replaced by the corresponding sample correlations $R_{ij}= S_{ij}/\sqrt{S_{ii}S_{jj}}$.
In Section \ref{sec:4}, we extend our results to hypercubic random matrices of the form $\sum_{t=1}^n \x_t \otimes \cdots \otimes \x_t$ and we briefly discuss some statistical applications.
The proofs of the main results are presented in Sections \ref{sec:proof1} and \ref{sec:proofcov11}.

\section{Normal approximation to large deviation probabilities}\label{sec:largedev}\setcounter{equation}{0}

In this section we collect some precise \ld\ results for sums of 
independent \rv s. 
Throughout this section, $(X_i)$ is an iid \seq\ of mean zero, unit variance \rv s with generic element $X$, \ds\ $F$ and right 
tail $\ov F=1-F$. We define the corresponding partial sum process 
\beao
S_0=0\,,\qquad S_n=X_1+\cdots+X_n\,,\qquad n\ge 1\,.
\eeao
Consider iid copies $(S_n^{(i)})_{i\ge 1}$ of $S_n$.
We also introduce an integer \seq\ $(p_n)$ \st\ $p=p_n\to\infty$ as $\nto$.
We are interested in the limit behavior of the $k$ largest values
among $(S_n^{(i)})_{i=1,\ldots,p}$, in particular in the possible limit laws of the 
maximum $\max_{i=1,\ldots,p} S_n^{(i)}$.
More generally, we are interested
in the limit behavior of the \pp es $N_n$,
\beam\label{eq:conv}
N_n=\sum_{i=1}^p \vep_{d_p (S_n^{(i)}/\sqrt{n}-d_p)} \std N\,,\qquad \nto\,,
\eeam
toward  a
Poisson random \ms\ $N$ on $\bbr$ with mean \ms\ $\mu$
given by $\mu(x, \infty)=\ex^{-x}$, $x\in\bbr$. The \seq\ 
$(d_p)$ is chosen \st\  $p\,\ov \Phi(d_p)= p(1-\Phi(d_p))\to 1$ as $p\to\infty$
where $\Phi$ is the standard normal distribution function. In this paper, we work with
\beam \label{eq:d_n}
d_p=\sqrt{2\log p} - \dfrac{\log\log p+\log 4\pi}{2(2\log p)^{1/2}}\,.
\eeam
A motivation for this choice is that for an iid \seq\ $(X_i)$ with \df\ $\Phi$
we have 
\beao
\lim_{\nto} \P\Big(d_p \big(\max_{i=1,\ldots,p} X_i-d_p\big)\le x\Big)=\exp(-\ex^{-x})=\Lambda(x)\,,
\qquad x\in\bbr\,. 
\eeao
The limit \df\ is the standard Gumbel $\Lambda$; see Embrechts et al. \cite[Example~3.3.29]{embrechts:kluppelberg:mikosch:1997}.
\par 
By \cite[Theorem~5.3]{resnick:2007}, relation \eqref{eq:conv} is equivalent to the following limit relation
for the tails
\beam\label{eq:esesa1}
p\,\P\big(d_p\,(S_n/\sqrt{n}-d_p)>x\big) \to \ex^{-x}\,, \qquad\nto\,,\qquad x\in\bbr\,,
\eeam
and also to \con\ of the maximum of the random walks $(S_n^{(i)})_{i=1,\ldots,p}$ to the Gumbel 
\ds :
\beam\label{eq:esesa2}
\lim_{\nto}\P\Big(\max_{i=1,\ldots,p} d_p(S_n^{(i)}/\sqrt{n}-d_p)\le x\Big)=\Lambda(x)\,,\qquad x\in\bbr\,.
\eeam
Equations \eqref{eq:esesa1} and \eqref{eq:esesa2}  involve precise \ld\ \pro ies for the random walk $(S_n)$.
To state some results which are relevant 
in this context, 
we assume one of the following three moment conditions:
\begin{enumerate}
\item[\rm (C1)]
There exists $s>2$ \st\ $\E[|X|^s]< \infty$.
\item[\rm (C2)]
There exists an increasing differentiable \fct\ $g$ on $(0,\infty)$ 
\st\ $\E[\exp(g(|X|))]<\infty$,
$g'(x)\le \tau g(x)/x$ for sufficiently large $x$ and some $\tau<1$, and 
$\lim_{\xto}g(x)/\log x=\infty$.
\item[\rm (C3)] 
There exists a constant $h>0$ \st\ $\E[\exp(h\,|X|)]<\infty$.
\end{enumerate}
Note that the conditions \rm (C1)--\rm (C3) are increasing in strength. One has the implications \rm (C3) $\Rightarrow$ \rm (C2) $\Rightarrow$ \rm (C1).
The following result explains the connection between the rate of 
$p_n\to\infty$ in  \eqref{eq:esesa1} and the conditions (C1)--(C3) on the distribution of $X$.
\bth \label{thm:d_p}
Assume the standard conditions on $(X_i)$ and that 
$p=p_n\to\infty$ satisfies
\begin{itemize}
\item
$p=O(n^{(s-2)/2})$ if {\rm (C1)} holds.
\item
$p= \exp(o(g_n^2\wedge n^{1/3}))$ where
$g_n$ is the solution of the equation $g_n^2=g(g_n \sqrt{n})$,
if {\rm (C2)} holds.
\item
$p=\exp(o(n^{1/3}))$ if {\rm (C3)} holds.
\end{itemize}
Then we have 
\beqq\label{eq:lll}
p\,\P(S_n/\sqrt{n}> d_p+x/d_p) \sim p\, \ov \Phi (d_p+x/d_p)\to \ex^{-x},\quad \nto, \quad x\in\bbr\,.
\eeqq
\ethe
\bre\label{rem:not}
The proofs of these results follow from the definition of $d_p$ and 
precise \ld\ bounds of the type
\beam\label{eq:ss}
\sup_{0\le y\le \gamma_n}
\Big|\dfrac{\P(S_n/\sqrt{n}>y)}{\ov \Phi(y)}-1\Big|\to 0\,,\quad \nto\,,
\eeam
where $\gamma_n\to\infty$ are \seq s depending on the conditions (C1)--(C3).
If (C3) holds, one can choose $\gamma_n=o(n^{1/6})$ implying the 
growth rate $p=\exp(o(n^{1/3}))$.
 This follows from
Petrov's \ld\ result \cite[Theorem VIII.2]{petrov:1972}. Under (C2)
one can choose $\gamma_n=o(n^{1/6}\wedge g_n)$ implying the growth rate
$p= \exp(o(g_n^2\wedge n^{1/3}))$.
This follows from S.V. Nagaev's
\cite[Theorem~3]{nagaev:1969}. Under (C1) he also derived $\gamma_n=\sqrt{(s/2-1)\log n}$ in \cite[Theorem~4]{nagaev:1969}. The best possible range 
under (C1) is $\gamma_n=\sqrt{(s-2)\log n}$; see Michel \cite[Theorem~4]{michel:1974}.
\par

The aforementioned \ld\ results cannot be improved in general
unless additional conditions are assumed. For example, under (C3) if the 
cumulants of $X$ of order  $k=3,\ldots,r+2$ vanish then
\eqref{eq:lll} holds  for $p=\exp(o(n^{(r+1)/(r+3)}))$.
This follows from the fact that one can choose 
$\gamma_n=o(n^{(r+1)/(2(r+3))})$; see  \cite[Theorem VIII.7]{petrov:1972}.
In Section VIII.3 of \cite{petrov:1972}  one also finds necessary and sufficient 
conditions for \eqref{eq:ss} to hold in certain intervals. As a matter of
fact, such conditions are not of moment-type. Therefore one cannot
expect that necessary and sufficient conditions for \eqref{eq:lll} for 
general \seq s $(p_n)$ can be expressed in terms of moments.
There is however a clear relationship between possible rates of $(p_n)$ 
and the existence of moments: the higher moments 
exist the larger we can choose $(p_n)$, but the growth cannot be arbitrarily fast. 
\par
In passing we mention a sharp \ld\ result for a \seq\ of iid \regvary\ \rv s $(X_i)$ with tail index $\alpha>2$, i.e., a generic element $X$ has tails
\beam\label{eq:regvarcond}
\P(\pm X>x)\sim p_{\pm}\,\dfrac{L(x)}{x^\alpha}\,,\qquad \xto\,,
\eeam
where $p_++p_-=1$ and $L$ is \slvary . 
Then, due to S.V. Nagaev's results in \cite{nagaev:1979}, one has \eqref{eq:ss} with $\gamma_n=\sqrt{c_1\log n}$ for $c_1<\alpha-2$, while for $\xi_n= \sqrt{c_2 \log n}$ and any $c_2>\alpha-2$, 
\beam\label{eq:nagresult}
\sup_{y>\xi_n} \Big|\dfrac{\P(\pm S_n/\sqrt{n}>y)}{n\,\ov F(y\,\sqrt{n})}- p_{\pm}\Big|\to 0\,,\qquad \nto\,.
\eeam
There exists a small but increasing literature on precise \ld\ results;
we refer to \cite{denisov:dieker:shneer:2008,rozovski:1993}
and the references therein.
\ere
Now consider iid copies $(S_n^{(i)})_{i\ge 1}$ of $(S_n)$.
The following result is an immediate con\seq\ of Theorem~\ref{thm:d_p}.
\bth\label{lem:einmahl}
Assume the conditions of Theorem~\ref{thm:d_p}. 
Relation \eqref{eq:lll} is equivalent 
to either of the following two limit relations:
\beqq\label{eq:4r}
\P\Big(d_p\max_{i=1,\ldots,p} (S_n^{(i)}/\sqrt{n}-d_p)\le x\Big)\to \Lambda(x)\,,\qquad x\in\bbr\,,\qquad \nto\,. 
\eeqq
and 
\beam\label{eq:resnick}
N_n=\sum_{i=1}^p\vep_{d_p\,(S_n^{(i)}/\sqrt{n}-d_p)}\std N=\sum_{i=1}^{\infty} \vep_{-\log \Gamma_i}\,,\qquad \nto\,.
\eeam
where $\Gamma_i= E_1+\cdots+E_i$, $i\ge 1$, and $(E_i)$ is iid standard exponential, i.e., $N$ is a Poisson random \ms\ with mean \ms\ $\mu(x,\infty)=\ex^{-x}$,
$x\in\bbr$.
\ethe
\begin{proof}
Following Resnick \cite[Theorem~5.3]{resnick:2007}, \eqref{eq:resnick}
and \eqref{eq:lll} are equivalent. Moreover, a continuous mapping argument
implies that, if $N_n\std N$, then
\begin{equation}\label{eq:trick}
\begin{split}
\P(N_n(x,\infty)=0)&= \P\Big(d_p\max_{i=1,\ldots,p} (S_n^{(i)}/\sqrt{n}-d_p)\le x\Big)\\
&\to \P(N(x,\infty)=0)= \P(-\log \Gamma_1\le x)= \exp( - \ex^{-x})\,.
\end{split}
\end{equation}
Moreover, if \eqref{eq:4r} holds a Taylor expansion argument shows that
\beao
\begin{split}
	\P\Big(d_p\max_{i=1,\ldots,p} (S_n^{(i)}/\sqrt{n}-d_p)\le x\Big) & = 
\Big(1- \dfrac{p\, \P(S_n/\sqrt{n}> d_p+x/d_p)}{p}\Big)^p \\ & \to \exp(-\ex^{-x}) \,, \qquad \nto \,,
\end{split}
\eeao
holds \fif\ \eqref{eq:lll} does.
\end{proof}

This means that in case of iid points $(S_n^{(i)})$ the convergence of the maximum is equivalent to the convergence of the point processes $(N_n)$. In general, the latter is a stronger statement. If $(N_n)$ converges, the distribution of the maximum can always be recovered using \eqref{eq:trick}.

\section{Main results}\label{sec:main}\setcounter{equation}{0}

\subsection{Point process convergence of a sample covariance matrix}\label{sec:gumbel}
We consider the
sample covariance matrix $\bfS=(S_{ij})_{i,j=1,\ldots,p}$ introduced in Section \ref{sec:model}.
The problem of showing limit theory for the associated point process is similar to Theorem~\ref{lem:einmahl} for iid random
walks $(S_n^{(i)})$. In contrast to the iid 
copies $(S_n^{(i)})$ in Section \ref{sec:largedev} here we deal
with $p(p-1)/2$ {\em dependent} off-diagonal entries of $\bfS$. Nevertheless, Theorem~\ref{thm:d_p} will again be a main tool
for proving these results. 
\par
Since the summands of $S_{ij}$ are iid products
$X_{it}X_{jt}$ we need to adjust the conditions (C2) and (C3) to this situation
while (C1) remains unchanged. 
\begin{enumerate}
\item[\rm (C2')]
There exists an increasing differentiable \fct\ $g$ on $(0,\infty)$ 
\st\ \linebreak $\E[\exp(g(|X_{11}X_{12}|))]<\infty$,
$g'(x)\le \tau g(x)/x$ for sufficiently large $x$ and some $\tau<1$, and 
$\lim_{\xto}g(x)/\log x=\infty$.
\item[\rm (C3')] 
There exists a constant $h>0$ \st\ $\E[\exp(h\,|X_{11}X_{12}|)]<\infty$.
\end{enumerate}

\begin{remark}{\em
By Lemma \ref{lem:mgf}, (C3') implies (C3). The reverse implication is not true. For example, if $X$ is standard exponential, which satisfies (C3), then $X_{11}X_{21}$ has Weibull-type tail with parameter $1/2$; see \cite{arendarczyk:debicki:2011}; which does not satisfy (C3'). By Lemma \ref{lem:mgf}, (C2') implies
$\E[\exp(g(|X|)]<\infty$. 
}\end{remark}

\bth\label{thm:cov}
Assume the standard conditions on $(X_{it})$ and that 
$p=p_n\to\infty$ satisfies: 
\begin{itemize}
\item
$p=O(n^{(s-2)/4})$ if {\rm (C1)} holds.
\item
$p= \exp(o(g_n^2\wedge n^{1/3}))$, where
$g_n$ is the solution of the equation $g_n^2=g(g_n \sqrt{n})$,
if {\rm (C2')} holds.
\item
$p=\exp(o(n^{1/3}))$ if {\rm (C3')} holds.
\end{itemize}
Define $\wt d_p =d_{p(p-1)/2}$. Then the following \pp\ \con\ holds:
\beao
N_n^S := \sum_{1\le i<j\le p} \vep_{\wt d_p(S_{ij}/\sqrt{n}-\wt d_p)}
\std N\,,
\eeao
where $N$ is the Poisson random measure defined in \eqref{eq:resnick}.
\ethe
The proof is given in Section~\ref{sec:proofcov}.
\subsubsection*{Some comments}
\begin{itemize}
\item The point process convergence in Theorem~\ref{thm:cov} remains valid if the standard conditions on $(X_{it})$ are relaxed to the following two conditions:
\begin{itemize}
\item The columns $\x_1,\ldots, \x_n$ of the matrix $(X_{it})_{i=1,\ldots,p; t=1 \ldots,n}$ are iid.
\item The random variables $X_{11},\ldots, X_{p1}$ are independent, with mean zero and unit variance, but they are not necessarily identically distributed.
\end{itemize}
The proof is the same as that of Theorem~\ref{thm:cov}. All results in Section~\ref{sec:main} hold under these relaxed conditions. For clarity of presentation and proof, all statements are presented under the standard conditions. 
\item
Theorem~\ref{thm:cov} can be extended by introducing additional
time stamps:
\beao
\sum_{1\le i<j\le p} \vep_{\big(\frac{(i,j)}{ p},\wt d_p(S_{ij}/\sqrt{n}-\wt d_p)\big)}
\std \wt N\,,\qquad \nto\,,
\eeao
on $\{(x_1,x_2): 0\le x_1\le 1\,,x_1\le x_2\}\times \bbr$
where $\wt N$ is a Poisson random \ms\ with mean \ms\ $\Leb\times \mu$.
This follows for example by using the techniques of \cite[Proposition~3.21]{resnick:1987}.
\item
Under any of the moment conditions {\rm (C2'),(C3')} one can choose 
$p\sim \gamma n$ for $\gamma>0$ in Theorem~\ref{thm:cov}. Under (C1), one needs the condition $\E[|X|^6]<\infty$ in order to guarantee $p=O(n)$. 
This is in agreement with the minimal moment requirement for the results on $W_n$ (see \eqref{eq:jiang2}).  
\end{itemize}

Next we consider the order statistics of $S_{ij}$, $1\le i<j\le p$:
\beao
\min_{1\le i<j\le p} S_{ij}=:S_{(p(p-1)/2)}\le \cdots \le S_{(1)}:=\max_{1\le i<j\le p}S_{ij}\,.
\eeao
Theorem~\ref{thm:cov} implies the convergence of the largest and smallest off-diagonal entries of $\bfS$. 
 \bco\label{cor:joint}
Under the conditions of Theorem~\ref{thm:cov} we have joint \con\
of the upper and lower order statistics: for any $k\ge 1$,
\beam
\wt d_p\big (S_{(i)}/\sqrt{n}-\wt d_p\big)_{i=1,\ldots,k}&\std& 
(-\log \Gamma_i)_{i=1,\ldots,k}\,,\label{eq:k1}\\
\wt d_p\big (S_{(i)}/\sqrt{n}+\wt d_p\big)_{i=p(p-1)/2,\ldots,p(p-1)/2-k+1}&\std& 
(\log \Gamma_i)_{i=1,\ldots,k}\,.\label{eq:k2}
\eeam
Moreover, the properly normalized maxima and minima are asymptotically independent, that is for any $x,y\in\bbr$ we have as $\nto$,
\beqq\label{eq:norming}
\P\Big(\wt d_p (S_{(1)}/\sqrt{n}-\wt d_p)\le x\,,\wt d_p
(S_{(p(p-1)/2)}/\sqrt{n}+\wt d_p)\le y\Big)
\to \Lambda(x) (1-\Lambda(-y) )\,.
\eeqq
\eco
\begin{proof} Relation \eqref{eq:k1} is immediate from $N_n^S\std N$
and the \cmt . The same argument works for \eqref{eq:k2} if one observes
that
\beao
\wt d_p\big (S_{(i)}/\sqrt{n}+\wt d_p\big)_{i=p(p-1)/2,\ldots,p(p-1)/2-k+1}=
- \wt d_p\big ((-S)_{(i)}/\sqrt{n}-\wt d_p\big)_{i=1,\ldots,k}
\eeao
where $(-S)_{(i)}$ is the ordered sample of $(-S_{ij})$. An application of  
\eqref{eq:k1} with $(S_{ij})$ replaced by $(-S_{ij})$ then yields \eqref{eq:k2}.
\par
Now we consider joint \con\ of the maximum and the minimum: for $x,y\in \R$,
\beao 
	\lefteqn{ G_n(x,y)} \\
	&= &\P\Big(\wt d_p (S_{(1)}/\sqrt{n}-\wt d_p)\le x\,,\wt d_p
(S_{(p(p-1)/2)}/\sqrt{n}+\wt d_p)> y\Big)\\
&= &  \P\Big(-\wt d_p+y/\wt d_p< S_{ij}/\sqrt{n}\le \wt d_p+x/\wt d_p\,\mbox{ for all $1\le i<j\le p$} \Big)\\
&= & 1- \P\Big(\bigcup_{1\le i<j\le p}\{S_{ij}/\sqrt{n}> \wt d_p+x/\wt d_p\}\cup\{ -S_{ij}/\sqrt{n}\ge \wt d_p-y/\wt d_p\}  \Big)
\,.
\eeao
Writing 
\beao
A_{ij}= \{S_{ij}/\sqrt{n}> \wt d_p+x/ \wt d_p\}\cup\{ -S_{ij}/\sqrt{n}\ge \wt d_p-y/\wt d_p\} \,,
\eeao
one can use the same arguments used for establishing $\P(N_n^S(B)=0) \to
\P(N(B)=0)$ in the proof of Theorem~\ref{thm:cov} to show that
\beao
G_n(x,y)\to \exp\big(-(\ex^{y}+\ex^{-x})\big)=\Lambda(x)\Lambda(-y)\,,\qquad \nto\,.
\eeao
Hence
\beao\lefteqn{
\P\Big(\wt d_p (S_{(1)}/\sqrt{n}-\wt d_p)\le x\,,\wt d_p
(S_{(p(p-1)/2)}/\sqrt{n}+\wt d_p)\le y\Big)}\\
&= &\P\Big(\wt d_p (S_{(1)}/\sqrt{n}-\wt d_p)\le x\Big)
- G_n(x,y)\\
&\to & \Lambda(x)- \Lambda(x)\Lambda(-y)
=\Lambda(x) (1-\Lambda(-y))\,,\qquad \nto\,.
\eeao
\end{proof}
\bre\label{cor:cov1}
An immediate con\seq\ is
\begin{equation*}
 \frac{S_{(1)}}{\sqrt{n \log p}} \cip 2 \quad \text{ and } \quad \frac{S_{(p(p-1)/2)}}{\sqrt{n \log p}} \cip -{2}\,.
\end{equation*}
\ere
\bre
If $\E[|X|^{s}]<\infty$ for some $s>4$ and $\var(X^2)>0$, we conclude from Theorem~\ref{lem:einmahl} that for $p=O(n^{(s-4)/4})$,
\beam\label{eq:rera}
\sum_{i=1}^p\vep_{d_p\big((S_{ii}-n)/\sqrt{n \var(X^2)}-d_p\big)}\cid N\,.
\eeam
In particular, $\big(\max_{i=1,\ldots,p} d_p\big((S_{ii}-n)/\sqrt{n \var(X^2)}-d_p\big)\big)$
converges to a Gumbel \ds . We notice that $d_p\sim \sqrt{2\log p}$
while the normalizing and centering 
constants for $(S_{ij}/\sqrt{n})_{i\ne j}$, 
in \eqref{eq:norming} are $\wt  d_p\sim 2 \sqrt{\log p}$.
\par
Moreover, while we still have Gumbel \con\ for the maxima of the off-diagonal
elements $S_{ij}$ for suitable $(p_n)$ if  $\E[|X|^{s}]<\infty$
for some $s\in (2,4)$, the \pp\ \con\ 
in \eqref{eq:rera} cannot hold. Indeed, then an appeal to Nagaev's
\ld\ result \eqref{eq:nagresult} shows that, under the \regvar\ 
condition \eqref{eq:regvarcond} on $X$ with $\alpha\in (2,4)$,
\beao
\sum_{i=1}^p\vep_{a_{np}^{-2} (S_{ii}-n) } \cid N\,,
\eeao
where $N$ is Poisson random \ms\ on the state space $(0,\infty)$ 
with mean \ms\ $\mu_\alpha(x,\infty)=x^{-\alpha/2}$, $x>0$,
and $a_k$ satisfies $k\P(|X|>a_{k})\to 1$
as $\kto$.
In particular, the maxima of $(S_{ii})$ converge toward a standard 
 Fr\'echet \ds :
\beao
\P\Big(a_{np}^{-2}\max_{i=1,\ldots,p}(S_{ii}-n)\le x\Big)\to \Phi_{\alpha/2}(x)=\exp(-x^{- \alpha/2})\,,\qquad x>0\,.
\eeao
Assume \eqref{eq:regvarcond} on $X$ with $\alpha\in (2,4)$. If we construct a point process by choosing the normalization $a_{np}^2$ for the diagonal and off-diagonal entries, the contribution of the $(S_{ij})$ vanishes in the limit:
\beao
\sum_{i=1}^p\vep_{a_{np}^{-2}(S_{ii}-n)}+\sum_{1\le i<j\le p}\vep_{a_{np}^{-2} S_{ij}}\cid N\,.
\eeao
It is also proved in Heiny and Mikosch \cite{heiny:mikosch:2017:iid}
that the diagonal entries  $(S_{ii})$ of the sample covariance matrix
dominate the off-diagonal terms in operator norm, that is $\norm{\bfS-\diag(\bfS)}/\norm{\diag(\bfS)}\cip 0$ as $\nto$.
In turn, the \asy\ behavior of the largest eigenvalues of the sample
covariance matrix are determined by the corresponding largest values of
$(S_{ii})$. 

\ere

The techniques in this paper straightforwardly extend to other transformations of the points $(S_{ij})$. As an example, we provide one such result for the squares $(S_{ij}^2)$.
\bco
Assume the conditions of Theorem~\ref{thm:cov}. Then
\beao
N_n^{S^2} = \sum_{1\le i<j\le p}\vep_{0.5 S_{ij}^2/n- 0.5 \,\wt d_p^2-\log 2}
\eeao
converges to the Poisson random \ms\ $N$ described in Theorem~\ref{thm:cov}.
\eco
\begin{proof} One can follow the arguments in the proof of
 Theorem~\ref{thm:cov}. In order to show condition (i), observe that for $x\in \R$,
\begin{equation*}
\begin{split}
\E\big[N_n^{S^2}(x,\infty)\big] &=\tfrac{p(p-1)}{2}\,\P\Big( \frac{S_{12}^2}{2n}- \frac{\wt d_p^2}{2}-\log 2>x \Big)\\
&=\tfrac{p(p-1)}{2}\,\P\Big(\Big|\frac{S_{12}}{\sqrt{n}}\Big|>\sqrt{2 (x+\log 2) +\wt d_p^2}\Big)\\
& \sim  p^2\,\ov \Phi \Big( \sqrt{2 (x+\log 2) +\wt d_p^2} \Big) 
\to \ex^{-x}\,, \qquad \nto\,.
\end{split}
\end{equation*}
\end{proof}

\subsection{Point process convergence of a sample correlation matrix}\label{sec:corr}
Based on Theorem~\ref{thm:cov} we can also derive 
point process convergence for the sample correlation matrix 
$\bfR=(R_{ij})_{i,j=1,\ldots,p}$ defined in
\eqref{eq:dcov} and  \eqref{eq:dcorr}. 

\bth\label{thm:corr}
Assume the standard conditions on $(X_{it})$ and that 
$p=p_n\to\infty$ satisfies: 
\begin{itemize}
\item
$p=O(n^{(s-2)/ 4})$ if {\rm (C1)} holds.
\item
$p= \exp(o(g_n^2 \wedge n^{1/3}))$ where
$g_n$ is the solution of the equation $g_n^2=g(g_n \sqrt{n})$ if {\rm (C2')} holds.
\item
$p=\exp(o(n^{1/3}))$, if  \rm (C3') holds. 
\end{itemize}
Then the following \pp\ \con\ holds,
\beao
N_n^R := \sum_{1\le i<j\le p} \vep_{\wt d_p(\sqrt{n} R_{ij}-\wt d_p)}
\std N\,,
\eeao
where $N$ is the Poisson random \ms\ defined in \eqref{eq:resnick}.
\ethe
The proof is given in Section~\ref{sec:proofcorr}.
\par
The results for the order statistics of $R_{ij}$, $1\le i<j\le p$:
\beao
\min_{1\le i<j\le p} R_{ij}=:R_{(p(p-1)/2)}\le \cdots \le R_{(1)}:=\max_{1\le i<j\le p} R_{ij}\,,
\eeao
carry over from those for the order statistics of $(S_{ij})$.
\bco\label{cor:corr}
Under the conditions of Theorem~\ref{thm:corr} we have joint \con\
of the upper and lower order statistics: for any $k\ge 1$,
\beao
\wt d_p\big (\sqrt{n} R_{(i)}-\wt d_p\big)_{i=1,\ldots,k}&\std& 
(-\log \Gamma_i)_{i=1,\ldots,k}\,,\\
\wt d_p\big (\sqrt{n} R_{(i)}+\wt d_p\big)_{i=p(p-1)/2,\ldots,p(p-1)/2-k+1}&\std& 
(\log \Gamma_i)_{i=1,\ldots,k}\,.
\eeao
Moreover, for any $x,y\in\bbr$,
\begin{equation*}
\lim_{\nto} \P\Big(\wt d_p (\sqrt{n} R_{(1)}-\wt d_p)\le x\,,\wt d_p
(\sqrt{n} R_{(p(p-1)/2)}+\wt d_p)\le y\Big)
= \Lambda(x) (1-\Lambda(-y))\,.
\end{equation*}
and 
\begin{equation*}
\sqrt{\frac{n}{\log p}} R_{(1)} \cip 2 \quad \text{ and } \quad  \sqrt{\frac{n}{\log p}} R_{(p(p-1)/2)} \cip -{2}\,.
\end{equation*}
\eco

\section{Extensions and applications}\setcounter{equation}{0}\label{sec:4}

\subsection{Extensions} \setcounter{equation}{0}

In this section, we extend our results for the point processes constructed from the off-diagonal entries of the sample covariance matrices 
$\bfS_n= \sum_{t=1}^n \x_t \x_t^\top$, where $\x_t$ are the $p$-dimensional columns of the data matrix $\X$. We introduce the hypercubic random matrices (or tensors) of order $m$:
\begin{equation}\label{eq:tensor}
\bfS^{(m)}= \bfS_n^{(m)}= \sum_{t=1}^n \underbrace{\x_t \otimes \cdots \otimes \x_t}_{m \text{ times}}  \,, \qquad m\in \N\,, \qquad n\ge 1\,.
\end{equation}
with entries
\begin{equation*}
S_{i_1, \ldots,i_m}^{(m)} = \Big( \sum_{t=1}^n X_{i_1 t} X_{i_2t} \cdots X_{i_mt} \Big)\,, \qquad {1\le i_1, \ldots, i_m \le p}\,.
\end{equation*}
It is easy to see that $\bfS^{(2)} = \bfS$ arises as a special case.

Next, we generalize the moment conditions (C2') and (C3') to the $m$-fold product $X_{11}\cdots X_{1m}$.
\begin{enumerate}
\item[\rm (C2(m))]
There exists an increasing differentiable \fct\ $g$ on $(0,\infty)$ 
\st\ \linebreak $\E[\exp(g(|X_{11}\cdots X_{1m}|))]<\infty$,
$g'(x)\le \tau g(x)/x$ for sufficiently large $x$ and some $\tau<1$, and 
$\lim_{\xto}g(x)/\log x=\infty$.
\item[\rm (C3(m))] 
There exists a constant $h>0$ \st\ $\E[\exp(h\,|X_{11}\cdots X_{1m}|)]<\infty$.
\end{enumerate}

The following result extends Theorem~\ref{thm:cov} to hypercubic matrices of order $m$. 
\bth\label{thm:covm}
Let $m\in \N$ and define $d_{p,m}=d_{\binom{p}{m}}$. Assume the standard conditions on $(X_{it})$ and that 
$p=p_n\to\infty$ satisfies: 
\begin{itemize}
\item
$p=O(n^{(s-2)/4})$ if {\rm (C1)} holds.
\item
$p= \exp(o(g_n^2\wedge n^{1/3}))$, where
$g_n$ is the solution of the equation $g_n^2=g(g_n \sqrt{n})$,
if {\rm (C2(m))} holds.
\item
$p=\exp(o(n^{1/3}))$ if {\rm (C3(m))} holds.
\end{itemize}
 Then the following \pp\ \con\ holds:
\beao
N_n^{(m)}= \sum_{1\le i_1< \cdots < i_m \le p} \vep_{d_{p,m}(S_{i_1, \ldots,i_m}^{(m)}/\sqrt{n}-d_{p,m})}
\std N\,,
\eeao
where $N$ is the Poisson random \ms\ defined in \eqref{eq:resnick}.
\ethe
The proof is given in Section \ref{sec:proofcovm}. 
Since $d_{p,m} \sim \sqrt{2m \,\log p}$, Theorem~\ref{thm:covm} implies the convergence of the largest and smallest off-diagonal entries of $\bfS^{(m)}$. 
\begin{corollary}{
Under the assumptions of Theorem~\ref{thm:covm}, we have, as $\nto$,
\begin{equation*}
\max_{1\le i_1< \cdots < i_m \le p} \frac{S_{i_1, \ldots,i_m}^{(m)}}{\sqrt{n \log p}} \cip \sqrt{2m} \quad \text{ and } \quad \min_{1\le i_1< \cdots < i_m \le p} \frac{S_{i_1, \ldots,i_m}^{(m)}}{\sqrt{n \log p}} \cip -\sqrt{2m}\,.
\end{equation*}
}\end{corollary}
Analogously to Corollary~\ref{cor:joint}, Theorem~\ref{thm:covm} yields the joint weak convergence of the off-diagonal entries of $\bfS^{(m)}$, thus extending Theorems 1 and 2 in \cite{jiang:xie:2019} on the asymptotic Gumbel property of the largest off-diagonal entry of $\bfS^{(m)}$.

\subsection{An application to threshold based estimators}

A fundamental task in statistics is the estimation of the population covariance or correlation matrix of a multivariate distribution. If the dimension $p$ becomes large, the sample versions $n^{-1} \bfS$ and $\bfR$ cease to be suitable estimators. Even for our simple model in Section \ref{sec:model}, i.e., when the population covariance and correlation matrices are the $p$-dimensional identity matrix $\bfI_p$, the estimators $n^{-1} \bfS$ and $\bfR$ are not asymptotically consistent for $\bfI_p$. This phenomenon was explored in \cite{heiny:2019} among many other papers.
Assuming $\E[X^4]<\infty$ and $p/n\to \gamma \in [0,\infty)$, \cite{heiny:2019} shows that, as $\nto$,
\beao
\sqrt{n/p}\, \norm{n^{-1} \bfS-\bfI_p} \cip 2 + \sqrt{\gamma} \quad \text{ and } \quad  \sqrt{n/p}\, \norm{\bfR-\bfI_p} \cip 2 + \sqrt{\gamma}\,.
\eeao
Note that $p$ is allowed to grow at a slower rate than $n$. It was also observed in \cite{heiny:2019} that 
\begin{equation}\label{eq:sdsdsdqw}
\sqrt{n/p}\, \norm{{ n^{-1}}\diag(\bfS)- \bfI_p}\cip 0\,.
\end{equation}
We would like to construct estimators $\widehat \bfS$, $\widehat \bfR$ based on $\bfS$ and $\bfR$, respectively, such that as $n\to \infty$,
\begin{equation}\label{eq:setew1}
\sqrt{n/p} \,\norm{{n^{-1}}\widehat\bfS- \bfI_p}\cip 0  \quad \text{ and } \quad   \sqrt{n/p}\, \norm{\widehat\bfR- \bfI_p}\cip 0.
\end{equation}
In view of \eqref{eq:sdsdsdqw}, we know that we are able to deal with the diagonal. A natural approach is to eliminate the smallest off-diagonal entries by thresholding.
 Bickel and Levina \cite{bickel:levina:2008a,bickel:levina:2008b} considered estimators of the form
\begin{equation}\label{eq:estimators}
\widehat \bfS=\big(S_{ij} \1(|S_{ij}|>n \, t_n) \big) \quad \text{ and } \quad \widehat \bfR=\big( R_{ij} \1(|R_{ij}|>t_n)\big)\,,
\end{equation}
for some threshold sequence $t_n\to 0$. Choosing $t_n=t_n(C)=C \sqrt{(\log p) /n}$ with a sufficiently large constant $C$, \cite[Theorem~1]{bickel:levina:2008a} shows \eqref{eq:setew1} for standard normal $X$. 
In view of Remark \ref{cor:cov1}, the order of the threshold perfectly matches the order of the largest off-diagonal entries. Based on our results, we provide a simple proof of \eqref{eq:setew1} for a more general class of distributions.

\begin{corollary}\label{thm:pspsp}
Assume $p/n\to \gamma \in [0,\infty)$ and the conditions of Theorem~\ref{thm:corr}.
Then the estimators $\widehat \bfR$, $\widehat \bfS$ in \eqref{eq:estimators} 
specified for $t_n(C)$, $C>2$,
satisfy relation \eqref{eq:setew1}.
\end{corollary}
\begin{proof}
The diagonal part is taken care of by \eqref{eq:sdsdsdqw} and the fact that $\diag(\bfR)=\bfI_p$. The off-diagonal entries of $\widehat \bfR$ and $\widehat \bfS$ asymptotically vanish in view of Remark \ref{cor:cov1} and Corollary \ref{cor:corr}, respectively.  
\end{proof}
Corollary \ref{thm:pspsp} shows that the order of the threshold $t_n(C)$ is not affected by the distributional assumption. 
Under (C1) we thus allow for $p=O(n^{(s-2)/4})$ provided $\E[|X|^s]<\infty$. 
For comparison,  Bickel and Levina \cite[p.~2585]{bickel:levina:2008a} showed the first limit relation in \eqref{eq:setew1} for the bigger threshold $t_n(C)= C p^{4/s}/\sqrt{n}$ and dimension $p=o(n^{s/8})$.

\subsection{An independence test}  If the data $(X_{it})$ is centered 
Gaussian with identical \ds\ the null hypothesis of independence is equivalent to $H_0: n^{-1} \E[\bfS]=\bfI_p$.
Based on \eqref{eq:jiang2}, Jiang \cite{jiang:2004b} proposed the following test of $H_0$ with significance level $\alpha\in (0,1)$:
\begin{equation*}
\Psi_{\alpha}= \1(n W_n^2-4 \log p + \log \log p\ge q_{\alpha})\,, 
\end{equation*}
where 
\begin{equation*}
q_{\alpha}= -\log (8\pi ) - 2 \log \log (1-\alpha)^{-1}
\end{equation*}
is the $(1-\alpha)$-quantile of the limiting non-standard Gumbel distribution. If $\Psi_{\alpha}=1$, we reject $H_0$. Properties of this test are studied in \cite{cai:liu:xia:2013}. 
\par
In view of Corollary \ref{cor:joint} we can propose a multitude of 
alternative tests based on the joint asymptotic distribution of the $k$ largest or smallest off-diagonal entries of $\bfS$ and $\bfR$, respectively. Under the conditions of Theorem~\ref{thm:cov} we have as $\nto$,
\begin{equation*}
\wt d_p\Big (\frac{(S_{(1)},\ldots,S_{(k)})}{\sqrt{n}}-\wt d_p\Big)\cid 
(-\log \Gamma_1, \ldots, -\log \Gamma_k)
\end{equation*}
and $\Gamma_i=E_1+\cdots+E_i$ for iid standard exponential random variables $(E_j)$. For $k\ge 1$ and $\alpha\in (0,1)$, consider a set $\mathcal{A}_k^{\alpha}\subset \R^k$ such that 
\begin{equation*}
\P\big((-\log \Gamma_1, \ldots, -\log \Gamma_k)\in \mathcal{A}_k^{\alpha}\big)=1-\alpha
\end{equation*}
and define the test $T(\mathcal{A}_k^{\alpha})$ by
\begin{equation*}
T(\mathcal{A}_k^{\alpha})=\1\Big(\wt d_p\Big (\frac{(S_{(1)},\ldots,S_{(k)})}{\sqrt{n}}-\wt d_p\Big) \notin \mathcal{A}_k^{\alpha}\Big)\,.
\end{equation*}
If $T(\mathcal{A}_k^{\alpha})=1$, we reject $H_0$. Then $T(\mathcal{A}_k^{\alpha})$ is an \asy\ independence test with significance level~$\alpha$.
\par
Convenient univariate test statistics can be constructed from 
spacings of $S_{(1)},\ldots,S_{(k)}$. An advantage of using spacings 
is that one avoids centering by $\wt d_p$. For example, consider for some $k\ge 2$,
\beao
T_k^{(1)}&=& \wt d_p\, (S_{(1)}-S_{(k)})/\sqrt{n}\,,\\
T_k^{(2)}&=& \wt d_p \max_{i=1,\ldots,k-1} (S_{(i)}-S_{(i+1)})/\sqrt{n}\,,\\
T_k^{(3)}&=& \wt d_p^2 \dfrac 1n \sum_{i=1}^{k-1}  (S_{(i)}-S_{(i+1)})^2\,.
\eeao
Recall the well-known fact that
\beao
\Big(\dfrac{\Gamma_1}{\Gamma_{k+1}},\ldots,\dfrac{\Gamma_k}{\Gamma_{k+1}}\Big)
\eqd \big(U_{(k)},\ldots,U_{(1)}\big)\,,
\eeao
where the right-hand vector consists of the order statistics of $k$ iid uniform
random variables on $(0,1)$.
Then we have
\beao
T_k^{(1)}&\std& \log \big(\Gamma_k/\Gamma_1\big)=\log 
\dfrac{\Gamma_k/\Gamma_{k+1}}{\Gamma_1/\Gamma_{k+1}}\eqd \log \big(U_{(1)}/U_{(k)}\big)\,,\\
T_k^{(2)}&\std& \max_{i=1,\ldots,k-1} \log (\Gamma_{i+1}/\Gamma_{i}) 
\eqd  \max_{i=1,\ldots,k-1} \log (U_{(k-i)}/U_{(k-i+1)})\,,\\
T_k^{(3)}&\std& \sum_{i=1}^{k-1} (\log (\Gamma_{i+1}/\Gamma_{i}))^2 
\eqd  \sum_{i=1}^{k-1} (\log (U_{(k-i)}/U_{(k-i+1)}))^2\,.\\
\eeao
Now, choosing $q_\alpha$ as 
the $(1-\alpha)$-quantiles of the limiting \rv s 
we have $T(\mathcal A_k^\alpha)=\1 ( T_k^{(i)} > q_\alpha)$, $i=1,2,3$.

\section{Proof of Theorem~\ref{thm:d_p}}\label{sec:proof1}\setcounter{equation}{0}
In view of Remark~\ref{rem:not} it suffices to prove the theorem under (C1). Throughout this proof we assume the standard conditions on $(X_{it})$.
\par
We start with a useful auxiliary result due to Einmahl 
\cite{einmahl:1989} (Corollary 1(b), p. 31, in combination 
with Remark on p. 32). 
\ble\label{lem:einmahl1} Consider independent $\bbr^d$-valued 
random vectors $\xi_1,\ldots,\xi_n$  with mean zero.
Assume that $\xi_i$, $i=1,\ldots,n$, has finite \mgf\ in 
some neighborhood of the 
origin and that the covariance matrix 
$\var(\xi_1+\cdots+\xi_n)=B_n \bfI_d$ where $B_n>0$ and $\bfI_d$ denotes the identity matrix.  Let $\eta_k$ be independent
$N(0,\sigma^2\var(\xi_k))$ random vectors independent of $(\xi_k)$, and $\sigma^2\in (0,1]$. Let $\xi_k^\ast=\xi_k+\eta_k$, $k=1,\ldots,n$, and write $p_n^\ast$ for the 
density of $B_n^{-1/2}(\xi_1^\ast+\cdots+\xi_n^\ast)$. Choose $\alpha\in (0, 0.5)$
\st\ 
\beam\label{eq:alpha}
\alpha \sum_{k=1}^n \E[|\xi_k|^3 \exp(\alpha |\xi_k|)]\le B_n\,,
\eeam
and write
\beao
\beta_n=\beta_n(\alpha)= B_n^{-3/2}\sum_{k=1}^n \E[|\xi_k|^3\exp(\alpha |\xi_k|)]\,,
 \eeao
where $|x|$ denotes the Euclidean norm.
If 
\beqq\label{eq:Tn}
|x|\le c_1 \alpha\,B_n^{1/2}\,,\qquad \sigma^2\ge -c_2 \beta_n^2\log \beta_n\,,\qquad B_n\ge c_3 \alpha^{-2}\,,
\eeqq
where $c_1,c_2,c_3$ are constants only depending on $d$, then
\beqq\label{eq:22b}
p_n^\ast(x)= \varphi_{(1+\sigma^2) \bfI_d } (x) \exp(\ov T_n(x))\;\mbox{ with }\;
|\ov T_n(x)|\le c_4\beta_n\,(|x|^3+1)\,, 
\eeqq
where $\varphi_{\boldsymbol{\Sigma}}$ is the density of a $N(0, \boldsymbol{\Sigma})$ random vector 
and $c_4$ is a constant only depending on $d$.
\ele

\begin{proof}[Proof under {\bf (C1)}] We proceed by formulating and proving various auxiliary results.
We will use the following notation: 
$c$ denotes any positive constant whose value is not of interest, sometimes we 
write $c_0,c_1,c_2,\ldots$ for positive constants whose value or size 
is relevant in the proof, 
\beao
\ov X_i&=&X_i\1(|X_i|\le n^{1/s})-\E[X\1(|X|\le n^{1/s})]\,,\qquad
\underline X_i= X_i-\ov X_i\,,\\
\ov S_n&=&\sum_{i=1}^n \ov X_i\,,\qquad \underline S_n= S_n-\ov S_n\,.
\eeao 
\par

Next we consider an approximation of the \ds\ of $\ov S_n$.
\ble\label{lem:1}
Let $\wt p_n$ be the density of 
\beao
n^{-1/2} \sum_{i=1}^n (\ov X_i +\sigma_n N_i)
\eeao 
where $(N_i)$ is iid $N(0,1)$, independent of $(X_i)$ and 
$\sigma_n^2=\var(\ov X)s_n^2$. If  $n^{-2c_6}\log n\le s_n^2\le 1$
with $c_6=0.5-(1-\delta)/s$ for arbitrarily small $\delta>0$, then the
relation
\beao
\wt p_n(x)= \varphi_{1+\sigma_n^2}(x)(1+o(1))\,,\qquad \nto\,,
\eeao
holds uniformly for $|x|= o(n^{1/6-1/(3s)})$. 
\ele
\begin{proof}
We apply Lemma~\ref{lem:einmahl1} to the iid \rv s 
$\xi_i=\ov X_i$, $i=1,\ldots,n$.
Notice that $\E[\ov X]=0$ and $B_n=\var(\ov S_n)=n\,\var(\ov X)$.
Choose
$\wt \alpha= c_5 n^{-1/s}$.
Then
\beao
 \wt \alpha \sum_{i=1}^n \E [|\ov X_i|^3\exp(\wt \alpha|\ov X_i|)]
&=& \wt \alpha\,n\, \E[|\ov X|^3\exp(\wt \alpha|\ov X|)]\\
&\le & c_5\, n^{1-1/s}\,\E[|\ov X|^3] \,\exp(2c_5)\\
&\le & 8 c_5\exp(2c_5) n^{1-\delta/s} \,\E[| X|^{2+\delta}]\,,
\eeao
where $\delta\in (0,1)$ is chosen \st\ $\E[|X|^{2+\delta}]<\infty$.
Hence \eqref{eq:alpha} is satisfied for $\alpha=\wt\alpha$ and sufficiently small $c_5$.
\par
Next choose 
\beam\label{eq:22ab}
\wt \beta_n &=&B_n^{-3/2}\sum_{i=1}^n \E \big[|\ov X_i|^3\,\exp(\wt \alpha |\ov X_i|)\big]\nonumber
=B_n^{-3/2}\,n\, \E[|\ov X|^3 \exp(\wt \alpha |\ov X|)]\nonumber\\
&\le& c\, B_n^{-3/2}n^{1+(1-\delta)/s}\,\E[|\ov X|^{2+\delta}]
\le c\,n^{-c_6}\,,
\eeam
where $\delta$ is chosen as above and $c_6=0.5-(1-\delta)/s$.
\par 
Next we consider \eqref{eq:Tn}. We can choose $x$ according to the restriction
\beam\label{eq:b1}
|x|\le c_1\, \wt \alpha\,B_n^{1/2}\sim c\,n^{1/2-1/s}\,.
\eeam
By \eqref{eq:Tn} and \eqref{eq:b1} we can choose $\sigma^2=\sigma_n^2$ according as 
\beam\label{eq:sigmaex}
1\ge \sigma_n^2\ge c\,\log n\,n^{-2c_6}\,.
\eeam
Moreover, $B_n\ge c_3\,\wt \alpha^{-2}$. An application of \eqref{eq:22b}
yields
\beao
\wt p_n(x)= \varphi_{1+\sigma_n^2}(x) \exp(\ov T_n(x))\quad\mbox{ for }\quad
|\ov T_n(x)|\le c_4\wt \beta_n(|x|^3+1)\,,
\eeao
but in view of \eqref{eq:22ab} and \eqref{eq:b1}, $\wt \beta_n(|x|^3+1)=o(1)$
uniformly for $|x|^3=o(\min(n^{0.5-1/s},n^{c_6}))= o(n^{0.5-1/s})$ for arbitrarily small $\delta>0$. That is, the remainder
term $|\ov T_n(x)|$ converges to zero, uniformly for the $x$ considered.
This proves the lemma.
\end{proof}
We add another auxiliary result.
\ble\label{lem:2}
Assume that $p=p_n\to\infty$ and $p= O(n^{(s-2)/2})$. 
Then for $x\in\bbr$, $c_6$ as in Lemma~\ref{lem:1}, an iid $N(0,1)$ sequence $(N_i)$ and $\sigma_n^{2}= c\,\log n\, n^{-2c_6}$, we have
\beao
p\,\P\Big(n^{-1/2}\sum_{i=1}^n (\ov X_i+\sigma_n N_i)>d_p +x/d_p\Big)\to \ex^{-x}\,,
\qquad \nto\,.
\eeao
\ele
\begin{proof}
Write $ y_n=\sqrt{ (s-2)\log n}$. By virtue of Lemma~\ref{lem:1}
we observe that for any $C>1$,
\begin{align*}
	& P_1=\P\Big(d_p +x/d_p<n^{-1/2}\sum_{i=1}^n (\ov X_i+\sigma_n N_i)\le y_n\Big)
\sim \int_{d_p+x/d_p}^{y_n} \varphi_{1+\sigma_n^2}(y)\,dy \,,\\
& P_2=  \P\Big(y_n <n^{-1/2}\sum_{i=1}^n (\ov X_i+\sigma_n N_i)\le C\,y_n\Big)
\sim \int_{y_n}^{Cy_n} \varphi_{1+\sigma_n^2}(y)\,dy \,.
\end{align*}
However, using Mill's ratio and the definition of $d_p$, we have that
\beao
\begin{split}
	p\,P_1&\sim  \ex^{-x}\,
\dfrac{\ov \Phi\Big(\dfrac{d_p+x/d_p}{\sqrt{1+\sigma_n^2}}\Big)}
{\ov\Phi(d_p+x/d_p)}- p \,\ov \Phi\Big(\dfrac{y_n}{\sqrt{1+\sigma_n^2}}\Big)
\\
&\sim \ex^{-x}\,
\exp\Big(0.5 (d_p+x/d_p)^2 \dfrac{\sigma_n^2}{1+\sigma_n^2}\Big)-\dfrac{1}{  \sqrt{2 \pi} \sqrt{(s-2) \log n}} p\, n^{-(s-2)/2}\,
\,,
\end{split}
\eeao
but the \rhs\ converges to $\e^{-x}$ since $(d_p+x/d_p)^2\sigma_n^2\sim d_p^2 \sigma_n^2=o(1)$, $p\, n^{-(s-2)/2}=O(1)$ and $(\log n)^2 n^{-2c_6}=o(1)$.
A similar argument shows $pP_2\to 0$.
\par
We also have
\beao
	\lefteqn{\P\Big(n^{-1/2}\sum_{i=1}^n (\ov X_i+\sigma_n N_i)> C\,y_n\Big) } \\
&\le &
\P\Big(n^{-1/2}\ov S_n > 0.5\,C\,y_n\Big)
+ \ov \Phi \big( 0.5 \,C y_n/\sigma_n\big)\\
&= & P_3+P_4\,.
\eeao
It is easy to see that $pP_4\to 0$.
We observe that 
\beao
|\ov X_i|&\le& n^{1/s}\big(1+o(n^{-1/s})\big)=c_n\,,\qquad a.s.\\
\var(\ov X)&\le &\E[X^2\1(|X|\le n^{1/s})]\le \var(X)=1\,.
\eeao 
We apply Prokhorov's inequality (Petrov  \cite[Chapter III.5]{petrov:1972}) 
for any $C>1$,
\beao
\lefteqn{p\,\P\big(\ov S_n>C \sqrt{n} \,y_n\big)  } \\
&  \le & p\,\exp\Big( - \dfrac{C\sqrt{(s-2)n\log n}}{2c_n}\,
\log \big(1+ \dfrac{C\sqrt{(s-2)n\log n}\,c_n} {2n\var(\ov X)}\big)\Big)\\
& \le &  p\,\exp\big(- \dfrac{C^2}{2}\dfrac{(s-2)\,\log n}{4}\big)\\
& =  & p\,n^{- C^2(s-2)/8}.
\eeao
The \rhs\ converges to zero for sufficiently large $C$.
This proves the lemma.
\end{proof}
Write $(X_{it})_{t\ge 1}$ for the iid \seq\ of the summands constituting $S_{n}^{(i)}$
and 
\beao
\ov S_{n}^{(i,N)}=\sum_{t=1}^n (\ov X_{it}+\sigma_n N_i)=:\ov S_{n}^{(i)}+\sigma_n \sqrt{n}\,\wt N_i\,,
\eeao
where $(\wt N_i)$ are iid standard normal \rv s independent of everything else.
Then by Lemma~\ref{lem:2},
\beao
\P\Big(\max_{i=1,\ldots,p} d_p\,(\ov S_{n}^{(i,N)}/\sqrt{n}-d_p)\le x\Big)\to \Lambda(x)\,,\qquad 
x\in\bbr\,,\qquad \nto\,. 
\eeao
We have 
\beao
d_pn^{-1/2}\max_{i=1,\ldots,p} \big|\ov S_{n}^{(i)}- \ov S_{n}^{(i,N)}|
&\le & d_p \max_{i=1,\ldots,p} |\sigma_n(\wt N_i-d_p)| + \sigma_n\,d_p^2\\
&\le &d_p\sigma_n \max_{i=1,\ldots,p} |\wt N_i-d_p| + \sigma_n\,d_p^2\\
&=& O_\P( d_p^2 \sigma_n)=o_\P(1)\,,\qquad \nto\,.
\eeao
Therefore 
\beqq\label{eq:springer}
\P\Big(d_p\big(\max_{i=1,\ldots,p} (\ov S_{n}^{(i)}/\sqrt{n}-d_p)\le x\Big)\to \Lambda(x)\,,\qquad 
x\in\bbr\,,\qquad \nto\,,
\eeqq
and the latter relation is equivalent to
\beam\label{eq:heklp0}
p\,\P(  \ov S_n/\sqrt{n}>d_p+x/d_p)\to \ex^{-x}\,, \qquad x\in\bbr\,,\qquad\nto\,.
\eeam
\par
Our next goal is to prove that we can replace $\ov S_n$ by $S_n$ in the latter relation. In view of the equivalence between \eqref{eq:springer} and 
\eqref{eq:heklp0} it suffices to show \eqref{eq:springer} with $\ov S_{n}^{(i)}$
replaced by $S_{n}^{(i)}$.
Therefore we will show that
\beao
\dfrac{d_p}{\sqrt{n}} \max_{i=1,\ldots,p}\big|\underline S_{n}^{(i)}\big|\stp 0\,.
\eeao
We have by the Fuk-Nagaev inequality \cite[p. 78]{petrov:1995}
for $y>0$ and suitable constants $c_0,c_1>0$,
\beam \label{eq:expo}
	\P\Big(d_p \max_{i=1,\ldots,p} |\underline S_{n}^{(i)}/\sqrt{n}|>y\Big) 
& \le &  p\,\P\big( |\underline{S}_n|>\sqrt{n}y/d_p\big)\nonumber\\
& \le &   c_0 n\,\E[|\underline{X}|^s] \big(\dfrac{\sqrt{n} y}{d_p}\big)^{-s}+
\exp\Big(-c_1 \dfrac{y^2}{d_p^2\,\var (\underline X)}\Big)\,.
\eeam
Using partial integration and Markov's inequality of order $s$, 
we find that $\var( \underline X)\le c\,n^{-0.5+1/(2s)}$
holds if $\E[|X|^s]<\infty$.
Combining this bound with the rate $p=O(n^{-(s-2)/2})$, we see that
$d_p^2\,\var (\underline X)\to 0$ and therefore
the exponential term in \eqref{eq:expo} vanishes. The polynomial term
in \eqref{eq:expo} converges to zero for the same reason. 
This proves \eqref{eq:springer} with $\ov S_{n}^{(i)}$ replaced by $S_{n}^{(i)}$
and finishes the proof of the theorem.
\end{proof}
\section{Proofs of sample covariance and correlation results}\label{sec:proofcov11}
\subsection{Proof of Theorem~\ref{thm:cov}}\label{sec:proofcov} \setcounter{equation}{0}

By Kallenberg's criterion for the convergence of simple point processes (see for instance \cite[p.~233, Theorem~5.2.2]{embrechts:kluppelberg:mikosch:1997})  it suffices to verify the
following conditions:
\begin{enumerate}
\item[(i)] For any $-\infty <a<b<\infty$, one has $\E[N_n^S(a,b]]\to \E[N(a,b]]=\mu(a,b]$ as $\nto$.
\item[(ii)] For $B=\cup_{i=1}^{\ell} (b_i,c_i]\subset (-\infty,\infty)$ with 
$-\infty<b_1<c_1<\cdots<b_{\ell}<c_{\ell}<\infty$, one has\\ $\P(N_n^S(B)=0)\to \P(N(B)=0)=\ex^{-\mu(B)}$ as $\nto$. 
\end{enumerate}
We start with (i). Note that $\mu(a,b]=\ex^{-a}-\ex^{-b}$. Since the assumptions of Theorem~\ref{thm:d_p} hold it follows from \eqref{eq:lll} (with $p$ replaced by $p(p-1)/2$),
that as $\nto$
\beao
\E[N_n^S(a,b]]= \dfrac{p(p-1)}{2} \P\big(\wt d_p +a /\wt d_p<S_{12}/\sqrt{n}<\wt d_p+b/\wt d_p)\to
\mu(a,b]\,.
\eeao

To show (ii), we consider
\beao
1- \P(N_n^S(B)=0)= \P\Big( \bigcup_{1\le i<j\le p} A_{ij}\Big)\,,
\qquad
A_{ij}=\{\wt d_p(S_{ij}/\sqrt{n} -\wt d_p)\in B\}\,.
\eeao
By an inclusion-exclusion argument we get for $k\ge 1$,
\begin{equation}\label{eq:ersgt}
\sum_{d=1}^{2k} (-1)^{d-1} W_d \le \P\Big( \bigcup_{1\le i <j \le p} A_{ij} \Big) \le \sum_{d=1}^{2k-1} (-1)^{d-1} W_d\,,
\end{equation}
where 
\begin{equation*}
W_d =\mbox{$\sum_{(I,J)\in I_d}$} \P( A_{i_1 j_1} \cap \cdots \cap A_{i_d j_d})=:\mbox{$\sum_{(I,J)\in I_d}$}\,q_{(I,J)}
\end{equation*}
and the summation runs over the set
\begin{equation*}
\begin{split}
I_d=\{ (I,J)=((i_1,j_1), \ldots , (i_d,j_d))  &\text{ such } \text{that }  1 \le i_t <j_t \le p, t=1, \ldots, d \,,\\
 & \text{ and } (i_1,j_1) < (i_2,j_2) < \cdots < (i_d,j_d)\}\,. 
\end{split}
\end{equation*}
In the definition of $I_d$, we use the lexicographic ordering of pairs $(i_s, j_s), (i_t,j_t)$:
\begin{equation*}
(i_s, j_s) < (i_t, j_t)  \quad \text{ if and only if } \quad  i_s<i_t \text{ or } (i_s=i_t \text{ and } j_s<j_t)\,.
\end{equation*}
A combinatorial argument yields 
\begin{equation}\label{eq:dsgds}
|I_d|=  \binom{\tfrac{p(p-1)}{2}}{d} \sim \frac{1}{d!} \Big( \frac{p^2}{2}\Big)^d\,, \qquad \nto\,.
\end{equation}
{\bf Proof of (ii) under (C1).}
Consider the set $\wh I_d$ consisting of all elements $(I,J)\in I_d$ such that all $i_t,j_t, t=1, \ldots, d$ are mutually distinct. For $(I,J) \in \wh I_d$  the random variables $S_{i_t,j_t}, t=1,\ldots d,$ are iid and therefore 
\begin{equation}\label{eq:etewt}q_{(I,J)}=\big( \P(A_{12}) \big)^d\,.
\end{equation}
For  $(I,J)\in I_d\backslash \wh I_d$ we write
\beao
\bfS_n^{(I,J)}= \big(S_{i_1j_1},\ldots,S_{i_d,j_d}\big)^\top
=\sum_{t=1}^n \big(X_{i_1t}{ X_{j_1t}},\ldots,X_{i_dt}X_{j_d t}\big)^\top=:\sum_{t=1}^n\xi_t\,,
\eeao
and also ${\bf1}=(1,\ldots,1)^\top\in\bbr^d$.
The iid $\bbr^d$-valued summands $\xi_t$ with generic element $\xi$ 
have mean zero and covariance matrix $ \bfI_d$.
We have
\beao
q_{(I,J)}
&=& \P\Big( n^{-1/2} \bfS_n^{(I,J)} \in \wt d_p \,{\bf1} + B^d/\wt d_p\Big)
\eeao
We will apply Lemma~\ref{lem:einmahl1} to $(\xi_t)$.  We will prove it under 
(C1); the proof under {(C2') and (C3')} is analogous; we will indicate some necessary changes. In this case, 
$\E[|\xi_i|^s]<\infty$ for some $s>2$. Write 
\beao
\ov \xi_t &=& \Big(\xi_t^{(l)}\1\big( |\xi_t^{(l)}|\le n^{1/s}\big)- 
\E[ \xi^{(l)}\1( |\xi^{(l)}|\le n^{1/s})]\Big)^\top_{l=1,\ldots,d}
\,,  \\
\underline\xi_t & = & \xi_t- \ov \xi_t\,,\\ 
\ov \bfS_n^{(I,J)}&=& \sum_{t=1}^n \ov \xi_t\,,\qquad \underline \bfS^{(I,J)}_n= \bfS_n^{(I,J)}-
\ov \bfS_n^{(I,J)}=\sum_{t=1}^n \underline \xi_t\,.
\eeao
\par
Proceeding as in the proof of Lemma~\ref{lem:1}, we obtain
the following result.
\ble\label{lem:11}
Let $\wt p_n$ be the density of 
\beao
n^{-1/2} \sum_{i=1}^n (\ov \xi_i +\sigma_n N_i)
\eeao 
where $(N_i)$ is iid $N(0,\bfI_d)$, independent of $(\xi_i)$ and 
$\sigma_n^2=\var(\ov \xi^{(1)})s_n^2$. If $n^{-2c_6}\log n\le s_n^2\le 1$
with $c_6=0.5-(1-\delta)/s$ for arbitrarily small $\delta>0$, then the
relation
\beao
\wt p_n(x)= \varphi_{(1+\sigma_n^2) \bfI_d}(x)(1+o(1))\,,\qquad \nto\,,
\eeao
holds uniformly for $|x|= o(n^{1/6-1/(3s)})$. 
\ele
Following the lines of the proof of Lemma~\ref{lem:2},
we obtain the following result:
\ble\label{lem:22}
Assume that $p=p_n\to\infty$ and $p^2= O(n^{(s-2)/2})$. 
Then for
$\sigma_n^{2}= c\,\log n\, n^{-2c_6}$ and an iid $N(0,1)$
\seq\ $(\wt N_i)$, uniformly for $(I,J)$ in $I_d$,
\beqq\label{eq:ext}
\Big(\dfrac{p^2}{2}\Big)^d\,q_{(I,J)}
\sim
\Big(\dfrac{p^2}{2} 
\P\Big(n^{-1/2} \sum_{t=1}^n \big(\ov \xi_t^{(1)} +\sigma_n \wt N_i\big)\in B\Big)\Big)^d
\sim (\mu(B))^d\,.
\eeqq
\ele
Finally, we need to prove that $\ov \bfS_n^{(I,J)}$ in \eqref{eq:ext}
can be replaced by $\bfS_n^{(I,J)}$. However, this follows in the same way as the 
corresponding steps in the proof of Theorem~\ref{lem:einmahl}.
Indeed, since we need to show that $n^{-1/2}\underline \bfS_n^{(I,J)}$ does not
contribute \asy ally to  $n^{-1/2}\bfS_n^{(I,J)}$ it suffices to prove this fact
for each of the components of  $n^{-1/2}\underline \bfS_n^{(I,J)}$.
\par
We conclude that as $\nto$
\beqq\label{eq:esrs}
W_d=\Big(\sum_{(I,J)\in I_d\backslash \wh I_d}+\sum_{(I,J)\in \wh I_d}\Big)q_{(I,J)} 
\sim \frac{1}{d!} \Big( \frac{p^2}{2}\Big)^d \big( \P(A_{12}) \big)^d
\sim \dfrac{(\mu(B))^d}{d!}.
\eeqq
We recall that \eqref{eq:ersgt} provides an upper and 
lower bound for $\P(N_n(B)=0)$. Letting first $\nto$ and then 
$k \to \infty$, thanks to \eqref{eq:esrs} we see that both bounds converge to the same limit. More precisely, we have
\begin{equation*}
\lim_{\nto} P(N_n(B)=0) = 1- \sum_{d=1}^{\infty} (-1)^{d-1} \frac{\big(\mu(B)\big)^d}{d!} = \sum_{d=0}^{\infty} \frac{\big(-\mu(B)\big)^d}{d!} = \e^{-\mu(B)}\,.
\end{equation*}
The proof of (ii) is complete.\\
{\bf Proof of (ii) under (C2'), (C3').} 
Write  $b_0=\min_{1\le q\le \ell} b_q$, $c_0= \max_{1\le q\le \ell} c_q$
and for $(I,J)\in I_d\backslash \wh I_d$, 
\beao
\wt S_n=S_{i_1j_1}+ \cdots + S_{i_dj_d}=\sum_{t=1}^n (X_{i_1t} X_{j_1t} + \cdots + X_{i_dt} X_{j_dt})\,.
\eeao
We have
\beam\label{eq:dfgh}
q_{(I,J)}  
&\le& \P \Big( \big( \frac{\wt S_n}{d \sqrt{n}} - \wt d_p \big)\, \wt d_p \in (b_0, c_0] \Big)\nonumber\\
&= &\P \Big( \sqrt{d} \big(b_0/{\wt d_p} +\wt d_p \big) < \frac{\wt S_n}{\sqrt{d\,n}}<\sqrt{d} \big(c_0/\wt d_p +\wt d_p \big) \Big)
\eeam
 Note that $\wt S_n/\sqrt{d}$ has iid summands with mean zero and unit variance.
Since $\sqrt{d}(c_0/\wt d_p +\wt d_p )=o(n^{1/6})$ under {(C3')} 
and $\sqrt{d}(c_0/\wt d_p +\wt d_p )=o(n^{1/6}\wedge g_n)$ under {(C2')}
applications of 
\cite[Theorem 1 in Section VIII.2]{petrov:1972}  and  \cite[Theorem~4]{nagaev:1969}, respectively, yield
\beao
q_{(I,J)}&\le & c\Big(\ov \Phi \big(\sqrt{d} (b_0/{\wt d_p} +\wt d_p)\big)
-\ov \Phi \big(\sqrt{d} (c_0/{\wt d_p} +\wt d_p)\big)\Big)= O(p^{-2d+\vep})\,.
\eeao
for an arbitrarily small $\vep>0$.
This shows that \eqref{eq:esrs} holds. Now one can proceed as under condition
(C1).

\subsection{Proof of Theorem~\ref{thm:covm}}\label{sec:proofcovm}

We proceed as in the proof Theorem~\ref{thm:cov} and show $(i),(ii)$ therein. 
For $-\infty <a<b<\infty$, it follows from \eqref{eq:lll} that as $\nto$
\beao
\begin{split}
	\E[N_n^{(m)} (a,b]] &= \binom{p}{m} \P\big(d_{p,m} +a /d_{p,m}<S_{12}/\sqrt{n}<d_{p,m}+b/d_{p,m}) \\
	& \to \ex^{-a}-\ex^{-b}\,.
\end{split}
\eeao
This proves condition (i). The proof of (ii) is completely analogous to the proof of Theorem~\ref{thm:cov}. The main difference is that $i<j$ needs to be replaced with $i_1<i_2< \cdots <i_m$. For example, instead of the index set $I_d$ whose elements are $d$ distinct $m$-tuples, with $|I_d|= \binom{\binom{p}{2}}{d}$; see \eqref{eq:dsgds}; one would get an index set $I_d^{(m)}$ of $d$ distinct $m$-tuples satisfying  $|I_d^{(m)}|= \binom{\binom{p}{m}}{d}$. We omit details.

\subsection{Proof of Theorem~\ref{thm:corr}}\label{sec:proofcorr}
 
First, assume $\var(X^2)=0$. 
Then $S_{ii}=n$ \as~ for all $i$ and hence $\sqrt{n}R_{ij} = S_{ij}/\sqrt{n}$ so that the claim follows immediately from Theorem~\ref{thm:cov}.
\par
In the remainder of this proof, we therefore assume $\var(X^2)>0$. 
By Theorem~\ref{thm:cov}, we already know that the point processes $N_n^S$
converge to a Poisson random \ms\ with mean measure $\mu(x,\infty]=\e^{-x}$, $x>0$. Our idea is to transfer the convergence of $N_n^S$ onto $N_n^R$. To this end, it suffices to show that (see \cite[Theorem 4.2]{kallenberg:1983}) $N_n^R-N_n^S \cip 0$ as $\nto$,
or equivalently that for any continuous function $f$ on $\R$ with compact support,
\begin{equation*}
\int f \dint N_n^R - \int f \dint N_n^S \cip 0\,, \quad \nto\,.
\end{equation*}
Suppose the compact support of $f$ is contained in $[K+\gamma_0, \infty)$ for some $\gamma_0>0$ and $K\in \R$. Since $f$ is uniformly continuous, $\omega(\gamma):= \sup \{|f(x)-f(y)|: x,y \in \R, |x-y| \le \gamma\}$ tends to zero as $\gamma\to 0$. 
We have to show that for any $\vep >0$,
\begin{equation}\label{eq:dgse}
\lim_{\nto} \P \Big( \Big| \sum_{1\le i<j \le p} \Big( f\big(  ( \sqrt{n} R_{ij} -\wt d_p) \wt d_p \big)- f\big( (S_{ij}/ \sqrt{n} -\wt d_p) \wt d_p \big) \Big)\Big| >\vep \Big) =0\,.
\end{equation}
On the sets
\begin{equation}\label{eq:An}
A_{n,\gamma}= \Big\{ \max_{1\le i<j \le p} \big| \wt d_p ( \sqrt{n} R_{ij} - S_{ij}/ \sqrt{n}) \big|  \le \gamma \Big\}\, ,\quad \gamma \in (0, \gamma_0) \,,
\end{equation}
we have 
\begin{equation*}
\big|f\big(  ( \sqrt{n} R_{ij} -\wt d_p) \wt d_p \big)- f\big( (S_{ij}/ \sqrt{n} -\wt d_p) \wt d_p \big) \big| \le \omega(\gamma) \,\1( (  S_{ij}/\sqrt{n} -\wt d_p) \wt d_p >K)\,.
\end{equation*}
Therefore, we see that, for $\gamma \in (0, \gamma_0)$,
\begin{equation}\label{eq:angamma}
\begin{split}
\P &\Big( \Big| \sum_{1\le i<j \le p} \Big( f\big(  ( \sqrt{n} R_{ij} -\wt d_p) \wt d_p \big)- f\big( (S_{ij}/ \sqrt{n} -\wt d_p) \wt d_p \big) \Big)\Big| >\vep, A_{n,\gamma} \Big)\\
&\le \P \Big( \omega(\gamma)\, \#\{1\le i<j \le p : (S_{ij}/ \sqrt{n} -\wt d_p) \wt d_p >K \} > \vep \Big)\\
&\le \frac{\omega(\gamma)}{\vep} \E\big[ \#\{1\le i<j \le p : (S_{ij}/ \sqrt{n} -\wt d_p) \wt d_p >K \}\big]\\
&= \frac{\omega(\gamma)}{\vep} \underbrace{\frac{p(p-1)}{2} \P((S_{12}/ \sqrt{n} -\wt d_p) \wt d_p >K)}_{\to \e^{-K} \text{ by Theorem~\ref{thm:d_p}}}
\,.
\end{split}
\end{equation}
Moreover, we have
\begin{equation*}
\begin{split}
\P(A_{n,\gamma}^c) &= \P\Big( \max_{1\le i<j \le p} \big| \wt d_p ( \sqrt{n} R_{ij} - S_{ij}/ \sqrt{n}) \big|  > \gamma \Big)\\
&= \P\Big( \max_{1\le i<j \le p} \wt d_p \frac{|S_{ij}|}{\sqrt{n}} \Big|  \frac{n}{\sqrt{S_{ii} S_{jj}}} - 1 \Big|  > \gamma \Big)\,.
\end{split}
\end{equation*}
Since $\max_{1\le i<j \le p} (S_{ij}/ \sqrt{n} -\wt d_p) \wt d_p \to \Lambda$, we get that $\max_{1\le i<j \le p} \wt d_p \frac{|S_{ij}|}{\sqrt{n}} = O_{\P}(\wt d_p^2)$. Thus,
\begin{equation}\label{eq:ancomplement}
\lim_{\nto} \P(A_{n,\gamma}^c) =0 
\end{equation}
is implied by 
\begin{equation}\label{eq:lefter1}
\lim_{\nto} \P\Big( \wt d_p^2 \max_{1\le i<j \le p} \Big|  \frac{n}{\sqrt{S_{ii} S_{jj}}} - 1 \Big|  > \beta \Big)=0\,, \quad \beta>0\,.
\end{equation}
Then taking the limits $\nto$ followed by $\gamma \to 0^+$ in \eqref{eq:angamma} and \eqref{eq:ancomplement} establishes \eqref{eq:dgse}.

It remains to prove \eqref{eq:lefter1}.
By the law of large numbers, $|S_{ii}/n| \cas 1$ as $\nto$. We have
\beao
\lefteqn{\max_{1\le i<j \le p} \Big|  \frac{n}{\sqrt{S_{ii} S_{jj}}} - 1 \Big| }\\
& & = \Big( \frac{n}{\min_{1\le i<j \le p} \sqrt{S_{ii} S_{jj}}} - 1 \Big) \vee \Big(1- \frac{n}{\max_{1\le i<j \le p} \sqrt{S_{ii} S_{jj}}} \Big)\\
& & \le \max_{1\le i \le p} \Big|  \frac{n}{S_{ii}} - 1 \Big|
\eeao
so that \eqref{eq:lefter1} follows from 
\begin{equation}\label{eq:lefter2}
\lim_{\nto} \P\Big( \wt d_p^2 \max_{1\le i \le p} \Big|  \frac{S_{ii}}{n} - 1 \Big|  > \beta \Big)=0 \,, \quad \beta>0\,.
\end{equation}
We have
\begin{equation*}
\begin{split}
\P\Big( \wt d_p^2 &\max_{1\le i \le p} \Big|  \frac{S_{ii}}{n} - 1 \Big|  > \beta \Big) 
\le p\, \P \Big( \frac{1}{\sqrt{n}} \Big| \sum_{t=1}^n (X_{1t}^2-1) \Big| > \frac{\beta \sqrt{n}}{\wt d_p^2}\Big) := \Psi_n\,.
\end{split}
\end{equation*}
It remains to prove that $\Psi_n\to 0$ under each of the conditions (C1), (C2'), (C3').
\par
First, assume (C1). Thus we have $\E[|X|^s]<\infty$ and $p=O(n^{(s-2)/4})$ for some $s>2$. An application of Markov's inequality yields
\beao
\Psi_n&\le& c\,\wt d_p^s\,n^{-(s+2)/4}\, \E\big[|S_{11}-n|^{s/2}\big]\,.
\eeao
By \cite[Lemma A.4]{davis:mikosch:pfaffel:2016} one has 
\beao
\E\big[|S_{11}-n|^{s/2}\big] \le c \, n^{\max(1,s/4)}
\eeao
and therefore it is easy to conclude that $\Psi_n=O((\log n)^{s/2}n^{-(1/4) \min(s-2,2)}) \to 0$ as $\nto$.
\par
Next, assume condition (C3'). By \cite[Section VIII.4, No.~8]{petrov:1972}, we have for $0\le x \le n^{\alpha}/\rho(n)$ with $0<\alpha\le 1/6$ and $\rho(n)\to \infty$ arbitrarily slowly that
\begin{equation}\label{eq:gsess}
\P \Big( \frac{1}{\sqrt{n}} \Big| \sum_{t=1}^n (X_{1t}^2-1) \Big| > x\Big) \sim 2 \ov \Phi (x/\sqrt{\var(X^{2}}) )\,, \quad \nto\,,
\end{equation}
if $\E\big[\exp\big(|X_{11}^2-1|^{4\alpha/(2 \alpha +1)}\big)\big]<\infty$. We apply this result with $\alpha=1/6$. Then the latter moment requirement reads $\E\big[\exp\big(|X_{11}^2-1|^{1/2}\big) \big]<\infty$ which in view of Lemma \ref{lem:mgf} is implied by (C3'). 
By definition of $\wt d_p$ and $p=\exp(o(n^{1/3}))$, we have
\begin{equation}\label{eq:estd}
\frac{\sqrt{n}}{\wt d_p^2}\sim \frac{\sqrt{n}}{4 \log p} 
>\frac{n^{1/6}}{\rho(n)}
\end{equation}
for any $\rho(n)\to \infty$. Using \eqref{eq:estd}, applying Mill's ratio and \eqref{eq:gsess} yield for a sequence $\rho(n)\to \infty$ sufficiently slowly that as $\nto$
\begin{equation*}
\begin{split}
\Psi_n  &\le p\, \P \Big( \frac{1}{\sqrt{n}} \Big| \sum_{t=1}^n (X_{1t}^2-1) \Big| > \frac{n^{1/6}}{\rho(n)} \Big)\sim 2p \, \ov \Phi \Big(\frac{n^{1/6}}{\rho(n) \sqrt{\var(X^{2})}} \Big)\to 0\,.
\end{split}
\end{equation*}
\par
Finally, assume (C2') and $p=\exp( o(n^{1/3}\wedge g_n^2))$.
We can proceed in the same way as under $(C1)$.
By Lemma~\ref{lem:mgf}, we have $\E[\exp(g(|X|))]<\infty$.
For any $\rho(n)\to \infty$ we have  
\beao
\frac{\sqrt{n}}{\wt d_p^2}>\frac{n^{1/6}}{\rho(n)}\ge \frac{n^{1/6}\wedge g_n'}{\rho(n)}\,.
\eeao
An application of \cite[Theorem~3]{nagaev:1969} shows that $\Psi_n \to 0$. The proof is complete.

\begin{lemma}\label{lem:mgf}
Let $Z,Z'\ge 0$ be iid random variables, $h$ a positive constant and $g$ an increasing \fct\ on $(0,\infty)$ such that $\E[\exp(g(h Z Z'))]<\infty$. Then we have $\E[\exp(g(Z))]<\infty$.
\end{lemma}
\begin{proof}
If $Z$ is bounded, the claim is trivial. Otherwise there exists $\alpha>1/h$ such that $\P(Z\le \alpha)<1$. Writing $F$ for the distribution function of $Z$, we have
\begin{equation*}
\begin{split}
&\E[\e^{g(Z)}] (1-F(\alpha)) = \int_{\alpha}^\infty \E[\e^{g(Z)}] \dint F(t)\\
&\le \int_{\alpha}^{\infty} \E[\e^{g(Z\, ht)}] \dint F(t) \le \E[\e^{g(h Z Z')}].
\end{split}
\end{equation*}
This implies $\E[\e^{g(Z)}] \le \E[\e^{g(h Z Z')}]/(1-F(\alpha))<\infty$.
\end{proof}

\bibliography{s_arvix}

\end{document}